\documentclass[12pt]{amsart}

\usepackage{verbatim}

\newtheorem{theorem}{Theorem}
\newtheorem{lemma}{Lemma}
\newtheorem{definition}{Definition}
\newtheorem{proposition}{Proposition}

\newtheorem{corollary}{Corollary}

\def\d{{\rm d}}
\def\e{{\rm e}}

\begin{document}

\title{Ergodicity of Burgers' system}
\author{Szymon Peszat}
\address{Szymon Peszat,  Institute of Mathematics, Jagiellonian University, {\L}ojasiewicza 6, 30--348 Krak\'ow, Poland}
\email{napeszat@cyf-kr.edu.pl}

\author{Krystyna Twardowska}
\address{Krystyna Twardowska, Faculty of Applied Informatics and Mathematics, Warsaw University of Life Sciences--SGGW, ul. Nowoursynowska 159, 02-776 Warsaw, Poland}
\author{Jerzy Zabczyk}
\address{Jerzy Zabczyk, Institute of Mathematics, Polish Academy of Sciences, ul. \'Sniadeckich 8, 00-956 Warsaw, Poland}
\date{}

 \thanks{The work of Szymon Peszat    was supported by Polish National Science Center grant  DEC2013/09/B/ST1/03658. }

\begin{abstract}
We consider a stochastic version of a system of coupled two equations formulated by Burgers \cite{Burgers} with the aim to describe the laminar and turbulent motions of a fluid in a channel. The existence and uniqueness of the solution as well as the irreducibility property of such system were given by Twardowska and Zabczyk  \cite{Twardowska-Zabczyk1, Twardowska-Zabczyk2, Twardowska-Zabczyk3}.  In the paper the existence of a unique invariant measure is investigated. The paper generalizes the results of  Da Prato, Debussche and Temam \cite{DaPrato-Debussche-Temam},  and Da Prato and Gatarek \cite{DaPrato-Gatarek},  dealing  with one equation describing the turbulent motion only.
\end{abstract}


\maketitle
\section{Introduction}

Let $U=U(t)$ denote the \emph{primary} velocity of the fluid, parallel to the walls of the channel, and let  $v=v(t,x)$ denote the \emph{secondary} velocity of the turbulent motion. According to \cite{Burgers}, they satisfy the following system of equations
\begin{align}
\frac{\d U }{\d t} ( t)&=P-\nu U( t) -\int_0^1v^2( t,y) \d y,   \label{1} \\
\frac{\partial v }{\partial t}( t,x) &=\nu \frac{\partial ^2v }{\partial x^2}( t,x)+U( t) v( t,x) -\frac \partial {\partial x}( v^2( t,x))
\label{2}
\end{align}
for $t>0$ and $x\in (0,1)$. The system is considered with the initial and boundary conditions
\begin{equation} \label{3}
\begin{aligned}
U( 0) &=U_{0},\ v( 0,x) =v_0( x) ,\quad  x\in (0,1)\\
v( t,0) &=v( t,1) =0,\quad  t>0. 
\end{aligned}
\end{equation}

In \eqref{1} and \eqref{2}, $P$ is  a constant representing an exterior force, analogous to the mean pressure gradient in the hydrodynamic case, and $\nu =\frac{\mu }{\rho }>0$, where $\rho$ is the density and  $\mu$ is the viscosity of the fluid. We assume that $\rho$ and $\mu$ are constant.   The system is derived from the theory of turbulent fluid motion and has similar properties to  the Navier--Stokes equation, but is much simpler to study.

The existence and uniqueness of the global solution  was established  by D\l otko  \cite{Dlotko}, using the Galerkin method. The necessary and sufficient conditions on $P$ and $\nu $ such that the solutions to \eqref{1}--\eqref{3} satisfy $U(t)\rightarrow 0$ and $v(t)\rightarrow 0$ as $t\rightarrow +\infty $ are given in \cite{Cholewa-Dlotko} and \cite{Dlotko}.

In order to obtain nontrivial limiting behaviour  of solutions as $t\rightarrow +\infty$, the   following stochastic perturbation was  proposed 
\begin{equation}\label{4}
\d U(t) =\left(P-\nu U( t) -\int_0^1  v^2(t,y)\d y \right)\d t+g_0\left(U(t),v(t)\right)\d W_0(t),  
\end{equation}
\begin{equation}\label{5}
\begin{aligned}
\d v(t,x) &=\left(\nu \frac{\partial ^2v }{\partial x^2}( t,x)+U( t) v(t,x)-\frac \partial {\partial x}( v^2( t,x) ) \right)\d t    \\
&\quad+g_1\left(U(t),v( t,\cdot ) \right)(x)\d W_1(t,x) 
\end{aligned}
\end{equation}
with the initial and boundary conditions \eqref{3}. Above $W_0$ is a real-valued Wiener process,  $W_1$ is an independent of $W_0$  cylindrical Wiener process in $L^2:= L^2(0,1)$, see the next section,  $g_0\colon \mathbb{R}\times L^2\mapsto \mathbb{R}$  and $g_1\colon \mathbb{R}\times L^2\mapsto L^2 $ are Lipschitz continuous functions. 

The existence and uniqueness of the global solution  to \eqref{3}--\eqref{5}  in the space $H:= \mathbb{R}\times L^2$ was established  by Twardowska and Zabczyk  \cite{Twardowska-Zabczyk1, Twardowska-Zabczyk2,Twardowska-Zabczyk3}. The existence and uniqueness of the solution for the classical one dimensional stochastic Burgers equation driven by cylindrical Wiener process, as well as the existence of an invariant measure was established by Da Prato, Debussche and Temam \cite{DaPrato-Debussche-Temam}.  

As the main result, we prove the existence of an invariant measure to system \eqref{3}--\eqref{5}. In the proof we use  the Krylov--Bogolyubov theorem adapting   the method od Da Prato and Gatarek  \cite{DaPrato-Gatarek}, see also \cite{DaPrato-Zabczyk}. The main difficulty is  caused  by  the fact that since the noise is not an $\mathbb{R}\times L^2$-valued process we cannot use the It\^o formula.  Due to existence of  nonlinear diffusion terms and different structure of our system we cannot apply the Hopf--Cole transformation as for example in  \cite{E-Khanin-Mazel-Sinai}.  

A similar problem of an hydrodynamic  equation coupled with heat equation was studied by Ferrario  \cite{Ferrario}. In the considered there B\'ennard  problem:  a  viscous  fluid,  in  a  rectangular container, is heated from below and the top surface is taken at constant temperature. Heating the fluid, its  density changes and the gradient of density gives rise  to a motion of particles from the bottom to the top of the container.

\section{Preliminaries and formulation of the main result}

Let $L^2:=L^2(0,1)$ be the Hilbert space equipped with the  scalar product  $\langle \varphi ,\psi \rangle =\int_0^1\varphi(x)\psi (x)\d x$ and the corresponding norm $\| \cdot \|$.

Let $(\Omega ,\mathcal{F},(\mathcal{F}_t)_{t\ge 0 },\mathbb{P})$ be a filtered probability space. We assume that $(\Omega, \mathcal{F},\mathbb{P})$ is complete, the filtration is right-continuous and each $\mathcal{F}_t$ contains  all $\mathbb{P}$-null sets in $\mathcal{F}$. We consider the one-dimensional Wiener process $W_0(t)$,  $t\ge 0$, and a  cylindrical Wiener process $W_1(t,\cdot )$, $t\ge 0$,  in $L^2$. It can be defined as the sum 
$$
W_1(t,x)=\sum_{k=1}^\infty W^k(t)e_k(x),  
$$
where $(e_k)$ is any orthonormal basis of $L^2$ and $(W^k)$ are independent real-valued Wiener processes. The sum converges however not in $L^2$ but in any Hilbert space $V$ such that the imbedding  $L^2\hookrightarrow V$ is  Hilbert--Schmidt.  We assume that for all $t>s>0$, $k\in \mathbb{N}$, the increments  $W_0(t)-W_0(s)$ , $W^k(t)-W^k(s)$ are independent of the $\sigma$-field $\mathcal{F}_s$.    

From now on,  we assume that  $(e_k)$ is the  orthonormal basis of the eigenvectors of the Laplace operator on $L^2$ with Dirichlet boundary conditions, that is 
\begin{equation}\label{6}
e_k(x)=\sqrt{\frac 2\pi }\sin k\pi x,\qquad x\in (0,1),\  k=1,2,\dots . 
\end{equation}

\begin{definition}
A pair of processes $(U,v)$ is \emph{a weak solution} to problem \eqref{3}--\eqref{5} if and only if they are adapted  with  continuous trajectories  in $\mathbb{R}$ and $L^2$, respectively, and  such that  for arbitrary $t\geq 0$, 
\begin{align*}
U(t)&=U_0+tP-\nu \int_0^tU(s)\d s-\int_0^t\| v(s)\|^2\d s  \\
&+\int_0^tg_0\left(U(s),v(s)\right) \d W_0(s),\quad \mathbb{P}\text{-a.s.,}  
\end{align*}
and for arbitrary $\varphi \in C_0^\infty (0,1)$, 
\begin{align*}
\langle v(t),\varphi \rangle  &= \langle v_0,\varphi \rangle +\int_0^tU( s) \langle v(s),\varphi \rangle \d s\\
&\quad +\int_0^t\left\langle v(s),\nu \frac{\d  ^2\varphi }{\d  x^2} \right\rangle \d s+\int_0^t\left\langle v^2(s),\frac{\d  \varphi }{\d  x}\right \rangle \d s
\\
&\quad +\int_0^t\left\langle \varphi ,g_1\left(U(s),v(s, \cdot)\right) \d W_1(s)\right\rangle ,\qquad \mathbb{P}\text{-a.s.}  
\end{align*}

\end{definition}

We introduce now an  equivalent concept of the \emph{integral or mild solution}, see e.g.  \cite{DaPrato-Zabczyk1, DaPrato-Zabczyk}. Namely, let $(e_k)$ be the orthonormal basis of $L^2$ given by \eqref{6}, and let $S(t)$, $t\geq 0$, be the classical heat semigroup on $L^2$. Then, for any $\psi\in L^2$,
$$
S(t)\psi=\sum_{k=1}^\infty \e^{-\frac{\pi ^2}\nu k^2t}\langle \psi ,e_k\rangle e_k  
$$
with the convergence of the series in $L^2$. It is well known that the generator $A$ of the semigroup $S(t)$, $t\geq 0$, is identical with the second derivative operator $\frac{\d  ^2}{\d  x^2}$ on the domain $D(A)$ consisting of functions $\psi $ such that $\psi $ and $\frac{\d \psi }{\d  x}$ are absolutely continuous, $\psi,\ \frac{\d \psi }{\d  x},  \  \frac{\d  ^2\psi }{\d  x^2}\in L^2$,  and  $\psi(0)=\psi(1)=0$. 

A proof of the following lemma can be find e.g.  in \cite{DaPrato-Debussche-Temam, Twardowska-Zabczyk1, Twardowska-Zabczyk2, Twardowska-Zabczyk3}.
\begin{lemma}\label{Lemma1}
The operators $S(t)$, $t> 0$, can be extended linearly to the space of all distributions of the form $\frac {\d \psi}  {\d  x}$, $\psi\in L^1(0,1)$, in such a way that they  take values in $L^2$ and for all $\psi\in L^2$ and $t\ge 0$, 
$$
\left\| S(t)\frac {\d  \psi }{\d  x}\right\| \leq \| \psi\| _{L^1(0,1)}\left(\sum_{k=1}^\infty \frac{2\pi }{\sqrt{\nu }}k^2\e^{- \frac{2\pi ^2}\nu k^2t}\right)^{1/2}
\le \frac{C}{\sqrt{t}}\| \psi\| _{L^1(0,1)}. 
$$
\end{lemma}

\begin{definition}
A pair of continuous adapted processes $(U,v)$ with values in $\mathbb{R}$ and $L^2$, respectively, is said to be \emph{an integral or mild solution to problem} \eqref{3}--\eqref{5} if
$$
\begin{aligned}
U(t)&=\e^{-\nu t}U_0+\int_0^t\e^{-\nu (t-s)}\left(P-\| v(s)\| ^2\right)\d s
 \\
&\quad +\int_0^t\e^{-\nu (t-s)}g_0\left(U(s),v(s)\right)\d W_0(s),  
\end{aligned}
$$
and
$$
\begin{aligned}
v(t)&=S(t)v_0+\int_0^tS(t-s)\left( U( s) v(s) - \frac {\d v^2}{\d  x}(s)\right)\d s  \\
&\quad +\int_0^tS(t-s)g_1\left(U(s),v(s)\right)\d W_1(s).  
\end{aligned}
$$
\end{definition}

In the integral
$$
\int_0^tS(t-s)\frac {\d v^2} {\d x}(s)\d s
$$
we use the extension of the operator $S(t-s)$ described in Lemma \ref{Lemma1}.

\bigskip
Let $Z_T^p$, $p>1$, denote the space of all continuous adapted processes $X = (U,v)$  on $[0,T]$ with values on $\mathbb{R}\times L^2$ such that
$$
\left\| (U, v)\right\| _{p,T}=\| U\| _{1,p,T}+\| v\| _{2,p,T}<+\infty, 
$$
where 
$$
\begin{aligned}
\| U\| _{1,p,T}&=\left(\mathbb{E}\, \sup_{t\in [0,T]}\vert U(t)\vert  ^p\right)^{1/p},\\
\|v\|_{2,p,T}&= \left(\mathbb{E}\, \sup_{t\in [0,T]}\| v(t)\| ^p\right)^{1/p}.   
\end{aligned}
$$

The first part of the theorem below concerning the existence of the solution was proven  in \cite{Twardowska-Zabczyk1,Twardowska-Zabczyk2,Twardowska-Zabczyk3} and it is stated here only for the sake of completeness.  The second and the third parts concerning existence and uniqueness of the invariant measure are the main result of the present paper. Their proofs are presented in Sections \ref{SInvariant} and \ref{SSFeller-Ireducibility}, respectively. 
\begin{theorem}
Assume that the mappings  $g_0\colon \mathbb{R}\times L^2\mapsto \mathbb{R}$ and $g_1\colon \mathbb{R}\times L^2\mapsto L^2$ are bounded and Lipschitz  continuous. Then the following assertions are true:\\
$(i)$  For any initial values $U_0\in \mathbb{R}$ and $v_0\in L^2$,  system \eqref{3}--\eqref{5} has a unique weak solution in any $Z_T^p$-space, $T>0$, $p\ge 1$.  Moreover, \eqref{3}--\eqref{5} defines Feller Markov process on the Hilbert space $H:= \mathbb{R}\times L^2$. \\
$(ii)$ There is an invariant measure $\mu$  for the Markov family defined by system \eqref{3}--\eqref{5}. \\
$(iii)$ If $g_0$ and $g_1$ are separated from zero; i.e. there is a constant $C>0$ such that 
$$
\left\vert \frac{1}{g_0(r,\psi)}\right\vert  \le C, \quad \left\| \frac{1}{g_1(r,\psi)}\right\|\le  C \qquad \text{for all $r\in \mathbb{R}$ and $\psi \in L^2$},
$$
then the corresponding transition semigroup $(P_t)$ is strong Feller and irreducible, and the invariant measure $\mu$  is unique. Finally, for any initial value $(U_0,v_0)$ the law $\mathcal{L}(U(t),v(t))$ of the solution at time $t$ converges in the total variation norm to $\mu$ as $t\to +\infty$. 
\end{theorem}
\section{Auxiliary results}

To prove the existence of an invariant measure we shall apply the classical Krylov--Bogo\-lyubov criterion  in the form of the following proposition, see \cite{DaPrato-Zabczyk}.

\begin{proposition}\label{P1} \textbf{(Krylov--Bogo\-lyubov)}  Assume that  $X$ is a Feller Markov process on a Hilbert space $H$ and $H_0\subset H$ is a Hilbert space contained in $H$ with compact embedding. If for some X(0) and for  each $\varepsilon >0$ there exists a constant $M$ such that for all $t>0$, 
$$
\frac 1t\int_0^tP\left(\| X(s)\| _{H_0}\geq M\right)\d s\leq \varepsilon ,
$$
then there exists an invariant measure for $X$ in $H$.
\end{proposition}

Let us recall that  $\nu$ is the viscosity parameter. The eigenvalues of $A$ are $-\lambda_m = {\frac{\pi^2 m^2}{\nu}}$,  $m=1,\ldots$.   The operator $-A$ is self-adjoint and positive definite and for arbitrary $\alpha >0$ the fractional power $(-A)^{\alpha}$ is given by the formulae,
$$
(-A)^{\alpha}\psi = \sum_{m=1}^{\infty} \lambda_m^{\alpha}\psi_me_m\quad \text{for}\quad  \psi=\sum_{m=1}^{\infty} \psi_m e_m \in D(-A)^{\alpha}, 
$$
where $\psi \in D(-A)^{\alpha}$ if and only if
$$
\|(-A)^{\alpha}\psi \| =\left(\sum_{m=1}^{\infty} \lambda_m^{ 2\alpha} \psi _m^2 \right) ^{1/2} < +\infty.
$$
We set $H^{2\alpha} = D(-A)^{\alpha}$. 

In our  proof of the existence of an invariant measure to \eqref{3}--\eqref{5} we will show that the  assumptions of the proposition above are satisfied  for    $H=\mathbb{R}\times L^2$ and  $H_0=\mathbb{R}\times H^{\frac 14}$. 

Let $X(t)=\left (U(t), v(t)\right)$ be the solution to problem \eqref{3}--\eqref{5}.   We assume that $U(0)=0$ and $v_0=0$. Given  $L>0$ set 
\begin{equation} \label{7}
\begin{aligned}
Z_L (t)&:=\int_0^t\e^{-L(t-s)} S(t-s)g_1\left(U(s),v(s)\right)\d W_1(s),\qquad t>0, \\
Y_L(t)&:= \int_0^t\e^{-(L+\nu) (t-s)}g_0\left(U(s),v(s)\right)\d W_0(s),\qquad t>0.
\end{aligned}
\end{equation}
\begin{lemma}\label{L2}
Let $p\ge 2$ and $\varepsilon >0$. There exists $L_0\ge 0$ such that for $L\ge L_0$ we have
$$
\sup_{0\leq t<+\infty }\mathbb{E}\left( \left\| Z_L (t)\right\|_{H^\frac 14} ^p+ \left\vert Y_L(t)\right\vert ^p\right)  
\le \varepsilon .  
$$
\end{lemma}
\begin{proof} Let $\| B \|_{(HS)}$ denote the Hilbert--Schmidt norm of a linear operator $B\colon L^2\mapsto L^2$. In \eqref{7} we identify an element   $g_1\left(U(s),v(s)\right)$ of $L^2$ with the  multiplication operator 
$$
g_1\left(U(s),v(s)\right)\colon L^2\ni \psi \mapsto g_1\left(U(s),v(s)\right)\psi \in L^2. 
$$
Clearly, its operator norm 
$$
\left\| g_1\left(U(s),v(s)\right)\right\|_{L(L^2,L^2)} \le \sup_{r\in \mathbb{R}, \psi \in L^2} \|g_1(r,\psi)\|:= K. 
$$
Hence using the Burkholder  inequality  and denoting by $S_L(t) =\e^{-Lt}S(t)$,  we obtain 
\begin{align*}
\mathbb{E}\left\| Z_L (t)\right\|_{H^\frac 14} ^p &\le c_p \mathbb{E} \left( \int_0^t\left\|(-A)^{\frac 18}S_L(t-s) g_1\left(U(s),v(s)\right)\right \|_{(HS)}^2 \d s \right)^{\frac{p}{2}}\\
&\le c_pK^p \left( \int_0^{+\infty} \left\|(-A)^{\frac 18}S_L(s)\right \|_{(HS)}^2 \d s \right)^{\frac{p}{2}}\\
&\le c_pK^p \left( \int_0^{+\infty} \e^{-2Ls} \left\|(-A)^{\frac 18}S (s /2)S(s/2) \right \|_{(HS)}^2 \d s \right)^{\frac{p}{2}}.
\end{align*}
Since 
\begin{align*}
\left\|(-A)^{\frac 18}S( s/ 2)S( s/2) \right \|_{(HS)}&\le \left\|(-A)^{\frac 18}S( s/ 2) \right \|_{L(L^2,L^2)}\left\|S( s/2) \right \|_{(HS)}\\
&\le C s^{-\frac 1 8} \left\|S( s/2) \right \|_{(HS)},  
\end{align*}
we have 
\begin{align*}
\mathbb{E}\left\| Z_L (t)\right\|_{H^\frac 14} ^p &\le C' \left( \int_0^{+\infty} \e^{-2Ls } s^{-\frac 14} \sum_{k=1}^\infty \e^{-\frac{\pi ^2}\nu k^2s}
 \d s \right)^{\frac{p}{2}}.
\end{align*} 
Therefore the desired conclusion follows from the Lebesgue dominated convergence theorem and fact that 
$$
\int_0^{+\infty}  s^{-\frac 14} \sum_{k=1}^\infty \e^{-\frac{\pi ^2}\nu k^2s}\d s =
 \int_0^{+\infty} s^{-\frac 14} \e^{-s}\d s \sum_{k=1}^\infty \left(\frac{\pi ^2}\nu k^2\right)^{-1 + \frac 14} <+\infty. 
$$
Clearly, in the same way one can obtain the desired estimate for $Y_L$. 
\end{proof}
\begin{corollary}\label{C1}
Taking into account the  Sobolev embedding $H^{\frac14}\hookrightarrow L^4$,  see Theorem 3, p. 8 in \cite{Runst-Sickel}, we obtain 
$$
\lim_{L\to +\infty} \sup_{t\ge 0} \mathbb{E}\, \|Z_L(t)\|_{L^4}^p=0. 
$$
\end{corollary}

Let 
$$
\begin{aligned}
V_L(t)&:= v(t)-Z_L(t),\\
U_L(t)&:= U(t)-Y_L(t).
\end{aligned}
$$
Note that $V_L$ and $U_L$ are solutions to the following system of equations with random coefficients: 
$$
\begin{aligned}
\frac{\d U_L}{\d t}(t)&= P-\nu U_L(t)- \|V_L(t)+Z_L(t)\|^2 + L Y_L(t),\\ 
U_L(0)&=U_0
\end{aligned}
$$
and 
$$
\begin{aligned}
\frac{\partial  V_L}{\partial t}(t,x)&= \nu \frac{\partial ^2V_L}{\partial x^2} (t,x) + \left(U_L(t)+Y_L(t)\right) \left( V_L(t,x)+ Z_L(t,x)\right) \\
&\quad - \frac{\partial }{\partial x} \left( V_L(t,x)+ Z_L(t,x)\right)^2 +LZ_L(t,x),\\ 
V_L(0,x)&=v_0(x),\\
V_L(t,0)&=0=V_L(t,1),\qquad t>0. 
\end{aligned}
$$

Recall, see e.g. \cite{Runst-Sickel},  that $H^1_0$ is the Sobolev space  consisting of absolutely continuous functions $\psi $ such that $ \frac{\d \psi }{\d  x}\in L^2$,  and  $\psi(0)=\psi(1)=0$.  We can now state  the main result of this section. 
\begin{lemma} \label{L3}
Given $L$ write 
$$
y_L(t):= \|V_L(t)\|^2 + U_L^2(t), \qquad t\ge 0. 
$$
Then  there are independent of $L$  constants $\beta >0$ and  $ C>0$ such that 
$$
\frac{\d }{\d t} y_L(t)+ \beta \left[ y_L(t)+ \|V_L(t)\|_{H^1_0}\right] \le C y_L(t) \xi_L(t) + C\eta_L(t),  
$$
where 
$$
\begin{aligned}
\xi_L(t)&:=  \vert  Y_L(t)\vert  +  \|Z_L(t)\|^2   +\|Z_L(t)\| + \| Z_L(t) \| _{L^4}^{\frac 8 4},\\
\eta_L(t)&:= (1+L^2) \left( Y_L ^4(t)+   \|Z_L(t)\|_{L^4}^4\right)  +1. 
\end{aligned}
$$
\end{lemma}
\begin{proof} We have
\begin{align*}
\left\langle \frac{\partial V_L }{\partial t},V_L \right\rangle  &=\nu \left \langle  \frac{\partial ^2V_L }{\partial x^2},V_L \right\rangle  +U\left\langle V_L
+Z_L ,V_L \right\rangle \\
&-\left\langle \frac \partial {\partial x}(V_L +Z_L )^2,V_L\right\rangle 
+L \left\langle Z_L ,V_L \right\rangle.
\end{align*}
Therefore 
\begin{align*}
\frac 12\frac {\d}{\d t}\| V_L\|^2 &= -\nu \left\langle \frac{\partial V_L }{\partial x},\frac{\partial V_L }{\partial x}\right\rangle  +U\langle V_L ,V_L\rangle +U\langle Z_L ,V_L\rangle \\
&+\left\langle V_L ^2,\frac{\partial V_L }{\partial x}\right\rangle +2\left\langle V_L Z_L ,\frac{\partial V_L }{\partial x}\right\rangle  +\left\langle Z_L ^2,\frac{\partial V_L }{\partial x}\right\rangle  +L \langle  Z_L ,V_L \rangle\\
&= -\nu \left\langle \frac{\partial V_L }{\partial x},\frac{\partial V_L }{\partial x}\right\rangle  +U\langle V_L ,V_L\rangle +U\langle Z_L ,V_L\rangle \\
&+2\left\langle V_L Z_L ,\frac{\partial V_L }{\partial x}\right\rangle  +\left\langle Z_L ^2,\frac{\partial V_L }{\partial x}\right\rangle  +L \langle Z_L ,V_L \rangle, 
\end{align*}
since $\left \langle V_L ^2,\frac{\partial V_L }{\partial x}\right\rangle = 0$.  In other words we have 
\begin{align*}
& \frac 12\frac {\d}{\d t} \| V_L \| ^2+\nu \| V_L \| _{H_0^1}^2 \\
&=U\| V_L \| ^2+U\langle V_L ,Z_L\rangle   +2\int_0^1  V_LZ_L \frac{\partial V_L }{\partial x}\d x+\int_0^1  Z_L ^2\frac{\partial V_L }{\partial x}\d  x+L \langle  Z_L ,V_L \rangle.  
\end{align*}
Further we have
\begin{align*}
\left\vert \int_0^1Z_L ^2\frac{\partial V_L }{\partial x}\d x\right\vert  &\le \| Z_L ^2\|^{\frac 12}\left\| \frac{\partial V_L }{\partial x}\right\| =\| Z_L \| _{L^4}^2\| V_L \| _{H_0^1}  \\
&\le \frac \nu 4 \| V_L \| _{H_0^1}^2+\frac 1{\nu }\| Z_L \| _{L^4}^4.  
\end{align*}
Using an   interpolation estimate, see e.g.  \cite{Lions-Magenes},  we obtain 
\begin{align*}
2\left\vert  \int_0^1  V_L Z_L\frac{\partial V_L }{\partial x}\d x\right\vert  
&\le 2C_1 \| V_L \| ^{\frac 3 4} \| V_L \| _{H_0^1}^{\frac 54} \|Z_L\|_{L^4} \le \frac \nu 4 \| V_L \| _{H_0^1}^2 + C_2\| V_L \| ^2  \|Z_L\|_{L^4}^{\frac 83}
\end{align*}
Hence 
\begin{align*}
\frac 12\frac {\d}{\d t} \| V_L \| ^2+\frac {\nu}{2}  \| V_L \| _{H_0^1}^2&\le  U\| V_L \| ^2+U\langle V_L ,Z_L \rangle +  L \|Z_L\| \|V_L\|   \\
&+C_2  \|V_L \| ^2\| Z_L \| _{L^4}^{\frac 8 4}  +\frac 1 \nu \| Z_L \|_{L^4}^4. 
\end{align*}
Recall that $U= U_L+Y_L$ and
\begin{align*}
\frac 12\frac{\d U_L^2}{\d t}+\nu U_L^2&=U_L\left(P-\| V_L +Z_L \| ^2 + LY_L\right)\\
&= U_L\left( P-\| V_L \| ^2-2\langle V_L ,Z_L \rangle -\| Z_L \| ^2\right) +LU_LY_L
\\
&\le \frac{\nu}{2} U_L^ 2 + \frac{1}{2\nu} P^2 -U_L\| V_L \| ^2-2U_L\langle V_L ,Z_L\rangle -U_L\| Z_L \| ^2 \\
&\qquad +LU_LY_L.
\end{align*}
Therefore adding the above inequality we obtain 
\begin{align*}
&\frac 12\frac {\d}{\d t}\left[  \| V_L \| ^2+U_L^2\right ]+\frac{\nu}{2} \left[ \| V_L \|
_{H_0^1}^2+ U_L^2\right]  \\
&\le Y_L \| V_L\| ^2+(Y_L-U_L) \langle V_L ,Z_L \rangle +L \|Z_L\| \|V_L\|  \\
&+C_2  \|V_L \| ^2\| Z_L \| _{L^4}^{\frac 8 4}  +\frac 1 \nu \| Z_L \|_{L^4}^4 + \frac{1}{2\nu} P^2 -U_L\| Z_L \| ^2 +LU_LY_L.
 \end{align*}
Thus, using  the Young inequality $|ab|\le a^p/p+ b^q/q$, we obtain 
\begin{align*}
&\frac {\d}{\d t}\left [ \| V_L\| ^2+U_L^2\right]+\frac{\nu}{2}  \left[  \| V_L \|_{H_0^1}^2+U_L^2\right ]  \\
&\le  2 \vert Y_L\vert  \| V_L\| ^2+2 (Y_L-U_L) \langle V_L ,Z_L \rangle +C_3 L^2  \|Z_L\|^2   \\
&\quad +2C_2  \|V_L \| ^2\| Z_L \| _{L^4}^{\frac 8 4}  +\frac 2 \nu \| Z_L \|_{L^4}^4 + \frac 1\nu P +C_3\| Z_L \| ^4 +C_3L^2  Y_L^2\\
&\le 2 \vert Y_L \vert \| V_L\| ^2+2 \vert Y_L\vert \|Z_L\|\|V_L\| +   2 \vert U_L\vert \|V_L\| \|Z_L \|   +C_3 L^2  \|Z_L\|^2   \\
&\quad +2C_2  \|V_L \| ^2\| Z_L \| _{L^4}^{\frac 8 4}  +\frac 2 \nu \| Z_L \|_{L^4}^4 + \frac{1}{\nu}P +C_3 \| Z_L \| ^4 +C_3L^2  Y_L^2\\
&\le 2 \vert Y_L \vert \| V_L\| ^2+ \vert Y_L\vert ^2 +  \|Z_L\|^2 \|V_L\|^2  +   U_L^2   \|Z_L \| + \|V_L\|^2 \|Z_L\| \\
&\quad  +C_3 L^2  \|Z_L\|^2  +2C_2  \|V_L \| ^2\| Z_L \| _{L^4}^{\frac 8 4}  +\frac 2 \nu \| Z_L \|_{L^4}^4 \\
&\le + \frac{1}{\nu} P +C_3 \| Z_L \| ^4 +C_3L^2  Y_L^2\\
&\le C_4 \left[ \| V_L\| ^2+U_L^2\right] \left(Y_L(t) +  \|Z_L\|^2   +\|Z_L\| + \| Z_L \| _{L^4}^{\frac 8 4} \right) \\
&\quad +C_4\left( (1+L^2) \vert Y_L\vert ^2 + (1+L^2)  \|Z_L\|^2 +\| Z_L \|_{L^4}^4 + P\right)
\end{align*}
and the desired estimate follows from the Poincar\'e inequality  
$$
\pi ^2\| V_L \| ^2\le \| V_L \| _{H_0^1}^2
$$ 
and from the fact that $\|Z_L\|+\|Z_L\|^2 \le 2 \|Z_L\|^4_{L^4} +2$. 
 \end{proof}
 \begin{corollary}\label{C2}
 Let $X=(U,v)$ be the solution to problem \eqref{3}--\eqref{5} with initial conditions $U(0)=U_0\in \mathbb{R}$ and $v(0)=v_0\in L^2$. Then 
 $$
 \lim_{M\to +\infty} \limsup_{T\to +\infty} \frac 1T \int_0^T \mathbb{P}\left( \vert U(t)\vert + \|v(t)\|\ge M\right)\d t=0. 
 $$
 \end{corollary}
 \begin{proof} We use the idea of Da Prato and Gatarek  \cite{DaPrato-Gatarek}, see also Da Prato and Zabczyk \cite{DaPrato-Zabczyk}. Namely, let us fix $M\ge 2$.  Recall, see Lemma \ref{L3},  that $ y_L(t):= \|V_L(t)\|^2 + U_L^2(t)$. Given $L>0$ set 
 $$
 \zeta_L(t):= \log \left( y_L(t) \vee \frac{M}{2}\right). 
 $$
 Then
 $$
 \frac{\d}{\d t}\zeta_L(t)= \chi_{\{y_L(t)\ge M/2\}}\frac {1}{y_L(t)}\frac{\d}{\d t} y_L(t). 
 $$ 
Recall, see Lemma \ref{L3} that 
 $$
 \frac{\d }{\d t} y_L(t)+ \beta y_L(t)\le C y_L(t) \xi_L(t) + C\eta_L(t).  
$$
Multiplying both sides of the inequality above, by 
$$
\chi_{\{y_L(t)\ge M/2\}}\frac {1}{y_L(t)}\le \frac 2M,  
$$
we obtain 
$$
\frac{\d }{\d t} \zeta_L(t)+ \beta \chi_{\{y_L(t)\ge M/2\}}\le C\xi_L(t)+ \frac{2C}{M}\eta_L(t). 
$$
Hence, after integrating over $t$ and taking expectation, we get 
\begin{align*}
&\mathbb{E}\left(\zeta_L(T)-\zeta_L(0)\right) + \beta \int_0^T \mathbb{P}\left( y_L(t)\ge \frac{M}{2}\right)\d t\\
&\quad \le C\int_0^T \left( \mathbb{E}\,\xi_L(t)+ \frac{2 \, \mathbb{E}\, \eta_{L}}{M}\right)\d t. 
\end{align*}
Since $\zeta_L(T)-\zeta_L(0)\ge 0$ because $M\ge 2$, we have 
$$
\frac{1}{T} \int_0^T \mathbb{P}\left( y_L(t)\ge \frac{M}{2}\right)\d t\le \frac{C}{\beta T}\int_0^T \left( \mathbb{E}\,\xi_L(t)+ \frac{2\, \mathbb{E}\, \eta_{L}}{M}\right)\d t. 
$$
Let us fix an $\varepsilon >0$. Then, by Lemma \ref{L2} there is a  constant $L_0$  such that for any $L\ge L_0$, 
$$
\sup_{t\ge 0}  \mathbb{E}\,\xi_L(t)\le \varepsilon \qquad \text{and}\qquad \sup_{t\ge 0}  \mathbb{E}\,\eta_L(t)\le (1+L^2)\varepsilon +1.  
$$
Therefore, for $L\ge L_0$, 
$$
\frac{1}{T} \int_0^T \mathbb{P}\left( y_L(t)\ge \frac{M}{2}\right)\d t\le \frac{C }{\beta }\left[ \varepsilon + \frac{(1+L^2)\varepsilon +1}{M}\right]. 
$$
Since  
$$
U^2(t) + \|v(t)\|^2 = y_L(t)+ Y^2_L(t) + \|Z_L(t)\|^2, 
$$
we have by the Chebyshev inequality 
\begin{align*}
&\frac{1}{T} \int_0^T \mathbb{P}\left( U^2(t) + \|v(t)\|^2 \ge M \right)\d t\\
&\quad \le \frac{1}{T} \int_0^T\left[  \mathbb{P}\left( y_L(t) \ge \frac{M}{2} \right)+\mathbb{P}\left(  Y^2_L(t) + \|Z_L(t)\|^2 
\ge \frac{M}{2} \right)\right] \d t\\
&\quad \le \frac{C }{\beta }\left[ \varepsilon + \frac{(1+L^2)\varepsilon +1}{M}\right] + \frac{2}{M} \sup_{t\ge 0} \mathbb{E}\left[ Y^2_L(t) + \|Z_L(t)\|^2\right], 
\end{align*}
and the desired conclusion follows since by Lemma \ref{L2}, 
$$
\lim_{L\to +\infty} \sup_{t\ge 0} \mathbb{E}\left[ Y^2_L(t) + \|Z_L(t)\|^2\right] =0. 
$$
To be more precise, given $\varepsilon >0$ we first find $L$ such that 
$$
\sup_{t\ge 0} \mathbb{E}\left[ Y^2_L(t) + \|Z_L(t)\|^2\right]  \le \frac{\varepsilon}{2}, 
$$
and than $M>2$ such that 
$$
\frac{C }{\beta }\left[ \varepsilon + \frac{(1+L^2)\varepsilon +1}{M}\right] \le \frac{\varepsilon }{2}.
$$
\end{proof}
\section{Proof of the existence of an  invariant measure}\label{SInvariant}
Taking into account the Krylov-Bogolyubow criterion (see Proposition \ref{P1}) and Corollary \ref{C2} it is enough to show that the family of  laws $(\mathcal{L}(v(t))$, $t>0$ is tight in the space $H^{\frac 14}$.  Recall that  $v(t)=V_L(t) +Z_L(t)$. On the other hand by Lemma \ref{L3} we have  the inequality 
\begin{align*}
\beta \|V_L(t)\|^2_{H^1_0}\le C y_L(t) \xi_L(t) + C\eta_L(t)\le C\left[ y^2_L(t) + \xi^2_L(t) +\eta _L(t)\right]. 
\end{align*}
Since 
$$
y_L(t) = \|V_L(t)\|^2 + U_L^2(t)\le 2\left(  \|v(t)\|^2 + U^2(t)\right) + 2\left( \|Z_L(t)\|^2 + Y^2_L(t)\right),
$$
and $\|\cdot \|_{H^{\frac 1 4}}\le C_ 1 \|\cdot \|_{H^1_0}$ we have 
\begin{align*}
\|v(t)\|_{H^{\frac 14}}&\le  \|Z_L(t)\|_{H^{\frac 14}} + C_2 \left(  \|v(t)\|^2 + U^2(t)\right)  \\
&\quad+ C_2 \left( \|Z_L(t)\|^2 + Y^2_L(t)\right)+ C_2 \left[ \xi^2_L(t) +\eta _L(t)\right]^{\frac 12}\\
&\le C_2 \left(  \|v(t)\|^2 + U^2(t)\right)  + R_L(t), 
\end{align*}
where, by Lemma \ref{L2}, 
$$
R_L(t):=  \|Z_L(t)\|_{H^{\frac 14}}+ C_2 \left( \|Z_L(t)\|^2 + Y^2_L(t)\right)+ C_2 \left[ \xi^2_L(t) +\eta _L(t)\right]^{\frac 12}
$$
satisfies 
$$
\sup_{t>0}  \mathbb{E}\, R_L(t)<+\infty. 
$$
Hence 
\begin{align*}
\mathbb{P}\left( \|v(t)\|_{H^{\frac 14}}\ge M\right) &\le \mathbb{P}\left(\|v(t)\|^2 + U^2(t)  \ge \frac{M}{2C_2}\right)+ \mathbb{P}\left( R_L(t)\ge\frac{M}{2} \right). 
\end{align*}
By Corollary \ref{C2}, 
$$
\lim_{M\to +\infty} \limsup_{T\to +\infty} \frac 1 T \int_0^T  \mathbb{P}\left(\|v(t)\|^2 + U^2(t)  \ge \frac{M}{2C_2}\right)\d t=0.
$$
By Chebyshev's inequality 
\begin{align*}
&\lim_{M\to +\infty} \limsup_{T\to +\infty} \frac 1 T \int_0^T  \mathbb{P}\left(R_L(t) \ge \frac{M}{2}\right)\d t\\
&\quad \le 
\lim_{M\to +\infty}\frac{2\sup_{t>0}  \mathbb{E}\, R_L(t)}{M}=0. \qquad \square
\end{align*}
\section{Strong Feller and Irreducibility}\label{SSFeller-Ireducibility}

In this section we will show that if $g_0$ and $g_1$ are separated from $0$, then the semigroup defined by problem \eqref{3}--\eqref{5} is strong Feller and irreducible. 

Let us recall first some basic definitions and theorems. Assume that $P_t(z,\Gamma)$, $t\ge 0$, $z\in E$, $\Gamma \in \mathcal{B}(E)$ is a transition probability   on a Polish space $E$. Let $(P_t)$ be the corresponding transition semigroup. 
\begin{definition}
We call a transition probability  $P$ \emph{regular}  if the measures $P_t(u,\cdot)$ and $P_t(z,\cdot)$,  $u,z\in E$ and $t,s>0$  are mutually absolutely continuous. \\
We call $P$ \emph{strong Feller} if the semigroup $P_t\colon B_b(E)\mapsto C_b(E)$ for $t> 0$. We call it \emph{irreducible} if $P(t,z,\mathcal{O})>0$ for all $t>0$, $z\in E$ and non-empty open set $\mathcal{O}\subset E$. 
\end{definition}
The first part of the theorem below is due to Khas'minskii \cite{Khas'minskii}. The second part is due to Doob \cite{Doob}. For the proofs of thees two parts we refer the reader also to \cite{DaPrato-Zabczyk}. The third part contains results of Seidler \cite{Seidler} and Stettner \cite{Stettner}. 
\begin{theorem}
(i) If a transition probability  is strong Feller and irreducible then it is regular. (ii) If $P$ is stochastically continuous and regular transition probability and $\mu$ is an invariant measure with respect to $P$, than $\mu$ is unique, equivalent to all measures $P_t(z,\cdot)$ and for all $z\in E$ and $\Gamma \in \mathcal{B}(E)$, 
$$
\lim_{t\to +\infty} P_t(z,\Gamma) = \mu(\Gamma), 
$$
(iii) Assume that the Markov process $X$ with a transition function $P$ is right-continuous, strong Markov and  regular.  Assume that $\mu$ is an invariant-measure for $X$. Then  for any probability measure $\tilde \mu$ on $E$ we gave 
$$
\lim_{t\to +\infty} \|P_t^*\mu-\mu\|_{\mathrm  TV}=0. 
$$ 
\end{theorem}

Recall that the \emph{total variation norm} of  two probability measures $\mu$ and $\tilde \mu$ on $E$ is given as follows 
$$ 
\|\mu-\tilde \mu\|_{\mathrm  TV} = \sup_{\Gamma\in \mathcal{B}(E)} \left\vert \mu(\Gamma)-\tilde \mu(\Gamma)\right\vert. 
$$
\subsection{Proof of strong Feller property} \label{SSFeller} In this part we are going to apply some standard techniques. Therefore we present here only a sketch of the proof. First we need to approximate $g_0$ and $g_1$ by sequence of $C^\infty$ mappings  $g_0^n$ and $g_1^n$ such that the corresponding Lipschitz constants converge; i.e. 
$$
\|g_0^n\|_{\mathrm {Lip}}\to \|g_0\|_{\mathrm {Lip}}\quad \text{and} \quad   \|g_1^n\|_{\mathrm {Lip}}\to \|g_1\|_{\mathrm {Lip}}.
$$
We require also point convergence and the upper and lower bonds: 
$$
\vert g_0^n(r,\psi)- g_0(r,\psi)\vert \to 0\  \text{and} \  \| g_1^n(r,\psi)- g_1(r,\psi)\| \to 0,\  \forall\, r\in \mathbb{R}, \psi\in L^2. 
$$
and there is a constant $C\in (0,+\infty)$ such that for all $r\in \mathbb{R}$ and $\psi\in L^2$ all the following numbers 
$$
  \vert g_0^n(r,\psi)\vert , \quad  \frac{1}{\vert g_0^n(r,\psi)\vert } ,\quad \left \|g_1^n(r,\psi)\right \|,\quad  \left \|\frac{1}{g_1(r,\psi)}\right \|  
 $$
 are less or equal to $C$. In \cite{Peszat-Zabczyk} one can find a construction of such sequences. 
 
We need also to regularize (in fact localize) the mappings $\mathcal{N}(\psi)=\|\psi \|^2$ and $B(\psi)= -\frac{\partial }{\partial x}\psi^2$. We set 
\begin{equation}\label{N}
\mathcal{N}_n(\psi) = \mathcal{N}(\psi) \phi_n(\|\psi\|^2), 
\end{equation}
where $\phi_n\in C^\infty(\mathbb{R})$ is such that $0\le \phi_n(r)\le 1$, $|\phi'(r)|\le 2$ for $r\in \mathbb{R}$, $\phi_{n}(r)=1$ for $r\le n+1$ and $\phi_n(r)= r $ for $r\le n$.  In the same fashion we localize $B$,  
\begin{equation}\label{B}
B_n(\psi)= B\left(\psi \phi_n(\|\psi\|^2)\right). 
\end{equation}
Let $X_{n,m}^z(t)=(U^z_{n,m}(t), v^z_{n,m}(t))$ be the value at time $t$ of the solution to the problem obtained from \eqref{3}--\eqref{5} by replacing $g_i$ by $g_i^m$,  and $\mathcal{N}$ and $B$ by $\mathcal{N}_n$ and $B_n$. We denote   by $z=(U(0), v(0)$ the initial data. We take  for granted the existence and uniqueness of the solution to the modify problem on the state space $H=\mathbb{R}\times L^2$. 

Let 
$$
P^{n,m}_tf(z)= \mathbb{E}\, f(X^z_{n,m}(t)), \qquad z=(r,\psi)\in \mathbb{R}\times L^2, 
$$
be the transition semigroup to the modify problem. As in \cite{Peszat-Zabczyk} one can derive  the Bismut--Elworthy--Li formula
\begin{align*}
&D_r   P^{n,m}_tf((r,\psi))a  + D_\psi P^{n,m}_tf\left((r,\psi)\right)[\varphi] \\
&\quad = \mathbb{E} \, f(X^z_{n,m}(t)) Y_{n,m}(r,\psi,t, a, \varphi),  
\end{align*}
where: $a\in \mathbb{R}$ and $\varphi\in L^2$ are directions,  $D_r$ and  $D_\psi $ denote the derivatives in $r$ and $\psi$, 
\begin{align*}
Y_{n,m}(r,\psi,t, a, \varphi)&= Y^1_{n,m}(r,\psi,t)a +Y^2_{n,m}(r,\psi,t, \varphi),\\
Y^1_{n,m}(r,\psi,t)&:= \frac {1}{t} \int  _0^t \frac{1}{g_0(X^{(r,\psi)}_{n,m}(s))} \frac{\partial }{\partial r} U^{(r,\psi)}_{n,m}(s)\d W_0(s),\\
Y^2_{n,m}(r,\psi,t,\varphi)&:= \frac {1}{t} \int  _0^t \left \langle \frac{1}{g_1(X^{(r,\psi)}_{n,m}(s))} D_\psi U^{(r,\psi)}_{n,m}(s)[\varphi] , \d W_0(s)\right\rangle. 
\end{align*}
We left to the reader the verification of the fact  that there is an independent of $m$, $r$,  $\psi$, $a$,   and $\varphi$ constant $C(n, t)$ such that 
$$
\mathbb{E}\left \vert Y_{n,m}(r,\psi,t,a,\varphi)\right\vert \le C(n,t)\left[ \vert a\vert ^2 + \|\varphi\|^2\right]^{\frac{1}{2}} . 
$$
This ensures not  only the strong Feller property of $(P^{n,m}_t)$ but also gives the estimate  
$$
\|DP^{n,m}_tf(z)\|_{L(H,H)}\le C(n,t)\sup_{z\in H}\vert f(z)\vert. 
$$
Therefore $P_tf$ is Lipschitz continuous function on $H$, and 
$$
\|P^{n,m}_tf\|_{\textrm{Lip}}\le C(n,t). 
$$
Since $C(n,t)$ does not depend on $m$ it enables us to pass with $n\to \infty$ and prove  the strong Feller property of the transition semigroup $(P^n_t)$ for the problem, say problem $\Pi_n$,  obtained from \eqref{3}--\eqref{5} by replacing only $\mathcal{N}$ and $B$ by $\mathcal{N}_n$ and $B_n$. In fact, $P^nf$ is Lipschitz continuous and for any bounded measurable function $f$ we have 
$$
\|P^{n}_tf\|_{\textrm{Lip}}\le C(n,t). 
$$

We can now conclude by using a standard localization procedure. Namely let $X^z_n$ be the solution to problem $\Pi_n$ and let $X^z$ be the solution to problem \eqref{3}--\eqref{5}. Then, from the uniqueness of the solutions and since for $\psi\colon \|\psi\|\le n$ we have $\mathcal{N}_n(\psi)=\mathcal{N}(\psi)$ and $B_n(\psi)=B(\psi)$, we infer that  
\begin{equation}\label{Lok}
X^z_n(s)= X^z(s) \qquad \text{for $s<\tau_{z}^n$},
\end{equation}
where 
$$
\tau_z^n:= \inf\{ s\colon \|X^z(s)\|\ge n\}. 
$$
We can now repeate arguments from \cite{DaPrato-Gatarek, DaPrato-Zabczyk}. Namely, let $f$ be a bounded measurable function on $H=\mathbb{R}\times L^2$. We have 
\begin{align*}
\left\vert P_t^n f(z)-P_tf(z)\right\vert &= \left\vert \mathbb{E}\left[ f(X^z_n(t))-f(X^z(t))\right]\chi_{\{t\ge \tau_z^n\}}\right\vert \\
&\le 2 \sup_{z\in H}\vert f(z)\vert \, \mathbb{P}\left( t\ge \tau_z^n\right). 
\end{align*}
Since functions $P^n_tf$ are   continuous, it is enough to show that for  all $M>0$ and $t>0$ we have 
$$
\lim_{n\to +\infty} \sup_{\|z\|_H\le M} \mathbb{P} \left( t\ge \tau_z^n\right)=0.  
$$
This can be written equivalently as 
$$
\lim_{n\to +\infty} \sup_{\|z\|_H\le M} \mathbb{P} \left( \sup_{s\le t} \|X^z(s)\|_H \ge n\right)=0,
$$
which can be easily deuced from Lemma \ref{L3},  see \cite{DaPrato-Gatarek} for more details, or from the existence of solution in the spaces  $Z_T^p$, $p>1$, established in \cite{Twardowska-Zabczyk1, Twardowska-Zabczyk2, Twardowska-Zabczyk3}.

\subsection{Proof of irreducibility} We will follow \cite{DaPrato-Gatarek, Peszat-Zabczyk}. Namely, from \eqref{Lok} we infer that it is enough to prove the irreducibility for the problem $\Pi_n$  obtained from the original problem \eqref{3}--\eqref{5} by replacing $\mathcal{N}$ and $B$ by $\mathcal{N}_n$ and $B_n$, see formulae \eqref{N} and \eqref{B} from Section \ref{SSFeller}. The idea is to use the Girsanov theorem. Let $\mathcal{O}\not =\emptyset$ be an open set in $H$ and let $T>0$. Let $z= (U_0, v_0)$ be a fixed initial datum.  We are going to show that $P^n_T(z, \mathcal{O})>0$.  Let $\overline z=(\overline U, \overline v)\in \mathcal{O}$. Since the domain $D(A)$ is dense in $L^2$ we can assume that $\overline v\in D(A)$. We will show that for any $\varepsilon>0$, $P^n_T(z, B(\overline z,\varepsilon))>0$, where  $B(\overline z,\varepsilon)$ is the ball in $H$ of radius $\varepsilon $ and center at $\overline z$.  To explain the idea it will be convenient for us to  write our problem in the following form 
$$
\d X_n =\left[ \mathcal{A}X_n + F_n(X_n)\right]\d t + \sigma(X_n)\d W,
$$
with with properly chosen  $\mathcal A$, $F_n$ and $\sigma$. In particular  $\mathcal{A}= (A,-\nu )$. Given  $0<h<T$ define 
$$
f^{h}(s)= \begin{cases} 0&\text{for $s\le T-h$}, \\
h^{-1}\e^{\mathcal{A}(s-(T-h)}\left( \overline z- X_n(T-h)\right) - \mathcal{A}z &\text{for  $s> T-h$}. 
\end{cases}
$$
Note that the mapping $f^{h}$ is such that 
\begin{equation}\label{Poy}
X_n(T-h) + \int_{T-h}^T \e ^{\mathcal{A}(T-s)}f^{h,m}(s)\d s = \overline z. 
\end{equation}
Let 
$$
\zeta_{h,m}(s):= \sigma(X_n(s))^{-1} \left[ f^{h}(s)\right] \chi_{\{\|X_n(T-h)\|_H\le m\}}. 
$$
Then under the probability measure 
$$
\d \mathbb{P}^* = \exp\left\{-\int_0^T  \zeta_{h,m}(s)\d W(s)-\frac 12 \int_0^T \|\zeta_{h,m}(s)\|^2 \d s\right\} \d \mathbb{P}, 
$$
$$
W^*(t)= W(t)+ \int_0^t \zeta_{h,m}(s)\d s, \qquad t\le T,
$$
is a cylindrical Wiener process on $H$.  Therefore the law of $X_n$ is equivalent to the law of the solution $\tilde X_n$ to the following equation
$$
\d \tilde X_n=\left[  \mathcal{A}\tilde X_n + F_n(\tilde X_n)+ f^h\chi_{\{\|\tilde X_n\|_H\le m\}}\right] \d t +  \sigma(X_n)\d W.
$$
Taking into account, \eqref{Poy}, we have
$$
\tilde X_n(T)= \overline z + R_h(T)\quad \text{if $\|X_n(T-h)\|_H\le m$,}
$$
where 
$$
R_h(T):= \int_{T-h}^T \e^{\mathcal{A}(T-s)}F_n(\tilde X_n(s))\d s+ \int_{T-h}^T \e^{\mathcal{A}(T-s)}\sigma(\tilde X_n(s))\d W(s).  
$$
Obviously we have
$$
\lim_{m\to +\infty}\mathbb{P}\left( \|X_n(T-h)\|_H\le m\right)=1.
$$
Thus the proof will be completed as soon as we can show that for any $\varepsilon >0$ we have 
$$
\lim_{h\downarrow 0} \mathbb{P}\left( \|R_h(T)\|_{H}<\varepsilon\right)>0.
$$
Since  it can be done in the same way as in Da Prato and Gatarek \cite{DaPrato-Gatarek} and we leave it to the reader.


\begin{thebibliography}{10}
\bibitem{Adam} 
R.A. Adams, \emph{Sobolev Spaces}, Academic Press, New York, 1975.

\bibitem{Burgers} 
J.M. Burgers, \emph{Mathematical examples illustrating relations occurring in the theory of turbulent fluid motion}, Verh. Nerderl. Akad. Wetensch. Afd. Natuurk. 17 (1939), 1--53.

\bibitem{Cholewa-Dlotko} 
J.W. Cholewa and T. D\l otko, \emph{Global Attractors in Abstract Parabolic Problems}, Cambridge Univ. Press, Cambridge, 2000.

\bibitem{DaPrato-Debussche-Temam} 
G. Da Prato, A. Debussche and R. Temam, \emph{Stochastic Burgers' equation}, NoDEA Nonlinear Differential Equations Appl. 1 (1994), 389--402.

\bibitem{DaPrato-Gatarek} G. Da Prato and D. G\c {a}tarek, \emph{Stochastic Burgers equation with correlated noise}, Stochastics Stochastics Rep. 52 (1995), 29--41.

\bibitem{DaPrato-Zabczyk1} G. Da Prato and J. Zabczyk, \emph{Stochastic Equations in Infinite Dimensions}, Cambridge, 1992.

\bibitem{DaPrato-Zabczyk} G. Da Prato and J. Zabczyk, \emph{Ergodicity for Infinite Dimensional Systems}, Cambridge Univ. Press, Cambridge, 1996.

\bibitem{Dlotko} T. D\l otko, \emph{The one-dimensional Burgers' equation; existence, uniqueness and stability}, Zeszyty Naukowe Uniw. Jagiello\'n skiegoPrace Mat. 23 (1982), 157--172.

\bibitem{Doob}
J.L. Doob, \emph{Assymptotic properties of Markoff transition probabilities}, Trans. Amer. Math. Soc. 63 (1948), 394--421. 

\bibitem{E-Khanin-Mazel-Sinai} W. E, K. Khanin, A. Mazel and Ya. Sinai, \emph{Invariant measures for Burgers equation with stochastic forcing}, Ann. of Math. (2) 151 (2000), 877--960.

\bibitem{Ferrario} B. Ferrario, \emph{The B\'{e}nard problem with random perturbations: dissipativity and invariant measures}, NoDEA Nonlinear
Differential Equations Appl. 4 (1997), 101--121.

\bibitem{Khas'minskii}
R.Z. Khas'minskii, \emph{Ergodic properties of recurrent diffusion processes and stabilization of the solution to the Cauchy problem for parabolic equations}, Theor. Probability Appl.  5 (1960), 179--196. 

\bibitem{Lions-Magenes} 
J.L. Lions and E. Magenes, \emph{Non-Homogeneous Boundary Value Problems and Applications Vol. I}, Springer 1972.  

\bibitem{Peszat-Zabczyk}
S. Peszat and J.  Zabczyk, \emph{Strong Feller property and irreducibility for diffusions on Hilbert spaces},  Ann. Probab. 23 (1995), 157--172.

\bibitem{Seidler} J. Seidler, \emph{Ergodic behaviour of stochastic parabolic equations}, Czechoslovak Math. J. 47  (1997),  277--316.

\bibitem{Stettner} L. Stettner, \emph{Remarks on ergodic conditions for Markov processes on Polish spaces},  Bull. Polish Acad. Sci. Math. 42 (1994), 103--114.

\bibitem{Runst-Sickel} T. Runst, W. Sickel, \emph{Soblev Spaces of Fractional Orders, Nemytskij Operators, and Nonlinear Partial Differential Equations}, de Gruyter, 1996.

\bibitem{Twardowska-Zabczyk1} K. Twardowska, J. Zabczyk, \emph{A note on stochastic Burgers' system of equations}, Preprint No. 646, Polish Academy of Sciences, Inst. of Math., Warsaw, 2003, 1--32.

\bibitem{Twardowska-Zabczyk2} K. Twardowska, J. Zabczyk, \emph{A note on stochastic Burgers' system of equations}, Stochastic Anal. Appl. 22 (2004), 1641--1670.

\bibitem{Twardowska-Zabczyk3} K. Twardowska, J. Zabczyk, \emph{Qualitative properties of solutions to stochastic Burgers' system of equations}, in: Stochastic Partial Differential Equations and Appl. VII, Ed. G. Da Prato and L. Tubaro, Proceedings SPDE's and Applications, Levico, Italy, 2004; Lect. Notes Pure Appl. Math. 245, Chapman \& Hall, Boca Raton, 2006, 311--322.
\end{thebibliography}
\end{document}